\documentclass{amsart}
\usepackage{amssymb}
\usepackage{amsmath}
\usepackage{setspace}
\usepackage{mathtools}
\usepackage{verbatim}
\usepackage{lipsum}
\usepackage{amsthm}
\usepackage{tikz-cd}

\usepackage{hyperref}

\newtheorem{theorem}{Theorem}[section]
\newtheorem{THM}{Theorem}

\newtheorem{lemma}[theorem]{Lemma}

\newtheorem{proposition}[theorem]{Proposition}
\newtheorem{defn}[theorem]{Definition}

\newtheorem{example}[theorem]{Example}

\newcommand{\bb}[1]{\mathbb{#1}}

\newcommand{\cal}[1]{\mathcal{#1}}

\newcommand{\TF}{T_{\mathcal F}}
\newcommand{\TG}{T_{\mathcal G}}
\newcommand{\NF}{N_{\mathcal F}}
\newcommand{\KF}{K_{\mathcal F}}

\DeclareMathOperator{\Pic}{Pic}
\DeclareMathOperator{\Div}{Div}
\DeclareMathOperator{\Inv}{Inv}
\DeclareMathOperator{\QInv}{QInv}
\DeclareMathOperator{\codim}{codim}
\DeclareMathOperator{\Frob}{Frob}
\DeclareMathOperator{\coker}{coker}
\DeclareMathOperator{\supp}{supp}

\title{Hypersurfaces quasi-invariant by codimension one foliations}

\author[Pereira]{Jorge Vitório Pereira}
\address{J. V. Pereira, IMPA, Estrada Dona Castorina 110, 22460-320, Rio de Janeiro, Brazil}
\email{\href{mailto:jvp@impa.br}{jvp@impa.br}}
\urladdr{\href{http://www.impa.br/~jvp}{http://www.impa.br/~jvp}}

\author[Spicer]{Calum Spicer}
\address{C. Spicer, Department of Mathematics, Imperial College London, England}
\email{\href{mailto:calum.spicer@imperial.ac.uk}{calum.spicer@imperial.ac.uk}}
\urladdr{\href{https://sites.google.com/view/calumspicer/home}{https://sites.google.com/view/calumspicer/home}}

\subjclass[2010]{37F75, 14E30}

\keywords{Foliations, quasi-invariant divisors,  extremal rays}

\begin{document}

\begin{abstract}
We present a variant of the classical Darboux-Jouanolou Theorem. Our main result provides a
characterization of foliations which are pull-backs of foliations on surfaces by rational maps.
As an application, we provide a structure theorem for foliations on $3$-folds admitting an infinite
number of extremal rays.
\end{abstract}

\maketitle

\setcounter{tocdepth}{1}

\section{Introduction}
In this paper, we study codimension one singular holomorphic foliations
on projective manifolds. More specifically, we introduce and investigate the concept of quasi-invariant hypersurfaces: an
irreducible hypersurface $H$ is quasi-invariant by a foliation $\cal F$ if it is not $\cal F$ invariant, but
the restriction of the foliation $\cal F$ to $H$ is an algebraically integrable foliation, i.e. every leaf of $\cal F_{|H}$
is algebraic.

\subsection{Statement of the main result} Rational pull-backs of
foliations on projective surfaces provide natural examples with infinitely many quasi-invariant divisors. Our main result shows
that the existence of sufficiently many quasi-invariant hypersurfaces characterizes this class of foliations.

\begin{THM}\label{THM:quasiJouanolou}
Let $\cal F$ be a codimension one holomorphic foliation on a projective  manifold $X$.
If $\cal F$ admits infinitely many quasi-invariant hypersurfaces then   $\cal F$ is algebraically integrable, or $\cal F$ is a pull-back of a foliation
on a projective surface under a dominant rational map.
\end{THM}

Indeed, we prove a  stronger statement in Section \ref{S:proof}, see  Theorem \ref{T:quasiJouanolou}.

Theorem \ref{THM:quasiJouanolou} has to be compared with, and was inspired by, Darboux-Jouanolou's criterion for the algebraic integrability of codimension one foliations.
We take the opportunity to revisit this criterion and we provide a small improvement of it in Section \ref{S:Jouanolou}, see  Theorem \ref{T:Jouanolou}.

\subsection{A version in positive characteristic}
In \cite{Kim97}, an analogue of Darboux-Jouanolou's Theorem was proved for
codimension one foliations on  smooth projective varieties defined over  fields of positive characteristic,
under the assumption that in the ambient variety
every global differential $1$-form is closed. This extra assumption was
later proved to be superfluous by Brunella and Nicolau in \cite{BN99}.

We show that Brunella-Nicolau's argument can be adapted to prove the following analogue
of Theorem \ref{THM:quasiJouanolou}.

\begin{THM}\label{THM:charp}
Let $X$ be a smooth projective variety defined over a field $k$ of characteristic $p>0$, and let $\cal F$ be a codimension one foliation
on $X$. If $\cal F$ admits infinitely many quasi-invariant hypersurfaces then $\cal F$ is $p$-closed,  or $\cal F$ is a pull-back of a foliation
on a projective surface under a dominant rational map.
\end{THM}

The concepts used in the statement of Theorem \ref{THM:charp} are recalled/presented in Section \ref{S:charp review}.

\subsection{Structure of the cone of curves}
The second author established in \cite{Spicer17} a cone theorem which describes the structure of the  Kleiman-Mori cone of curves
in terms of  numerical properties of the canonical bundle $\KF$ of a codimension one foliation $\cal F$ on a projective
$3$-dimensional variety.  We were lead to the definition of quasi-invariant divisors while trying
to understand the implications of this result on the geometry/dynamics of the original foliation.
As a corollary of Theorem \ref{THM:quasiJouanolou} we obtain the following statement.

\begin{THM}\label{THM:cone}
Let $\cal F$ be a codimension one foliation with non-dicritical canonical singularities on a connected $3$-dimensional $\bb Q$-factorial
projective variety $X$ with at worst terminal singularities.
If the $\KF$-negative portion of the cone of curves contains
infinitely many extremal rays then (at least) one of the following assertions holds true.
\begin{enumerate}
\item the foliation $\cal F$ is algebraically integrable and the closure of a general leaf is rational; or
\item there exists a finite number of rational $\cal F$-invariant hypersurfaces containing all but
finitely many extremal rays.
\end{enumerate}
\end{THM}

The definitions of extremal rays and canonical singularities are recalled in Section \ref{S:cone}, where the precise statement
of the cone theorem is also recalled, cf. \ref{T:cone}.

\subsection{Acknowledgments}
The results presented here were achieved in a visit of J.V. Pereira to Imperial College London
followed by a  visit  of C. Spicer to FRIAS, Freiburg. We acknowledge the support of Paolo Cascini and Stefan Kebekus
who made these visits possible. J.V. Pereira was partially supported by Cnpq and FAPERJ.

\section{Foliations, invariant and quasi-invariant divisors}\label{S:Jouanolou}

\subsection{Foliations}
A foliation $\cal F$ on a projective manifold $X$ is determined by an involutive and saturated subsheaf
$\TF$ of the tangent sheaf of $X$. The sheaf $\TF$ is the tangent sheaf of $\cal F$. The dimension of $\cal F$ is the rank of
$\TF$, and its codimension is the corank of $\TF$ inside $T_X$. The singular set
of $\cal F$ is the singular set of the sheaf $T_X/\TF$. As we are assuming $\TF$ saturated inside $T_X$,
the singular set of $\cal F$ is of codimension at least two. Away from the singular set, the Frobenius Theorem
provides a decomposition of  sufficiently small open (in the Euclidean topology) subsets $U \subset X$
in level sets of   submersions $f: U \to  V\subset \mathbb C^q$, where $q = \codim \cal F$. Maximal analytic continuation
of a level set of $f$ inside $X- \text{sing}(\cal F)$ is a leaf of the foliation $\cal F$.

The dual of $\TF$ is the cotangent sheaf of $\cal F$ and will be denoted by $\Omega^1_{\cal F}$. As the notation suggests,
the local sections of $\Omega^1_{\cal F}$ can be interpreted as $1$-forms defined along the leaves of the foliation $\cal F$.

Alternatively, a foliation $\cal F$ can be presented by a subsheaf $N^*_{\cal F}$ of $\Omega^1_X$, which is integrable and
saturated.  The sheaf $N^*_{\cal F}$ is called the conormal sheaf of $\cal F$. Its dual will be denote by $N_{\cal F}$ and is
called the normal sheaf of $\cal F$. One can recover $\TF$ from $N^*_{\cal F}$ by considering the subsheaf of $T_X$ which annihilates
every element of $N^*_{\cal F}$. The natural analogous construction allows us to recover $N^*_{\cal F}$ from $\TF$.

\subsection{Invariant divisors}
Let $\mathcal F$ be a codimension one foliation on a projective manifold $X$. We will denote
by $\Inv(\mathcal F)$  the subgroup of $\Div(X)$ generated by irreducible $\cal F$-invariant hypersurfaces.

\begin{proposition}
If $D$ is a $\cal F$-invariant divisor then the restriction of differential forms to leaves
of $\cal F$ defines a morphism
\[
    H^0(X, \Omega^1_X(\log D)) \to H^0(X, \Omega^1_{\cal F}) \, .
\]
\end{proposition}
\begin{proof}
    For any local section $v \in \TF(U)$ of the tangent sheaf of $\cal F$, the contraction of $v$ with a section
    of $\Omega^1_X(\log D)$ is holomorphic due to the $\cal F$-invariance of $D$.
\end{proof}

Recall that a holomorphic $\cal F$-partial connection on  coherent sheaf $\mathcal E$ is a $\mathbb C$-linear
morphism $\nabla :  \mathcal E \to \mathcal E \otimes_{\mathbb C} \Omega^1_{\cal F}$ satisfying Leibniz rule, i.e.
\[
    \nabla(f \cdot \sigma)  = f \cdot \nabla(\sigma) + \sigma \otimes (df)_{|\TF}
\]
where $f$ is a local section of $\mathcal O_X$ and $\sigma$ is a local section of $\mathcal E$.

\begin{proposition}
If $D \in \Inv(\cal F)$  is a $\cal F$-invariant divisor then the line-bundle $\cal O_X(D)$ admits
a holomorphic $\cal F$-partial connection.
\end{proposition}
\begin{proof}
    The logarithmic differentials of local defining functions of $D$ define a flat logarithmic connection $\nabla$
    on $\cal O_X(D)$ with residue divisor equal to $D$. The $\cal F$-invariance of $D$ implies that
    the restriction of $\nabla$ to $\TF$ is holomorphic.
\end{proof}

It was observed by Bott that the normal sheaf of a foliation $\cal F$ admits a natural $\cal F$-partial connection,
the so called Bott's partial connection. For codimension one foliations  Bott's connection can be locally defined
as follows. Let $\omega$ be a local generator of $N^*_{\cal F}$. The integrability assumption implies that the existence
of a meromorphic $1$-form $\eta$ such that
\[
    d \omega = \omega \wedge \eta \, .
\]
The $1$-form $\eta$ is not unique, one can replace $\eta$ by $\eta + h \omega$ for an arbitrary meromorphic function $h$. Nevertheless,
if $v$ is a local section of $\TF$ then the contraction of $v$ with $\eta$ does not depend on $h$ and, furthermore, it is a holomorphic
function. In other words, the restriction of the  $1$-form $\eta$ to $\TF$ defines a local holomorphic section of $\Omega^1_{\cal F}$.
Bott's connection is the connection on $N_{\cal F}$ locally defined by $\eta_{|\TF}$.

\subsection{Jouanolou's Theorem}
If $X$ is a projective manifold then we will denote by $NS(X)$  the Néron-Severi group of $X$, i.e. $NS(X)$ is the group
of divisors modulo numerical equivalence.  The tensor product of
$NS(X)$ with $\mathbb C$ injects into $H^2(X,\mathbb C)$ and  is the natural target for the Chern class map
$c : \Div(X) \otimes \mathbb C \to NS(X)\otimes X$.

Below we present a small improvement on Darboux-Jouanolou Theorem, compare with  \cite{Jouanolou}, \cite{Ghys00} and \cite[Theorem 6.1]{Brunella00}.
It involves the concept of a transversely affine foliation. We recall that a codimension one foliation $\cal F$ is called transversely
affine, if there exists a flat meromorphic connection on the normal sheaf of $\cal F$ which, when restricted to the leaves
of $\cal F$ coincides with  Bott's $\cal F$-partial connection. For a thorough discussion of this concept, and
other equivalent definitions the reader can consult \cite{MR3294560}.

\begin{theorem}\label{T:Jouanolou}
Let $\cal F$ be a codimension one foliation on a projective  manifold $X$.
If $k$ is the number of $\cal F$-invariant hypersurfaces and $\ell$ is the integer $\dim NS(X) + h^0(X,\Omega^1_{\cal F}) - h^0(X,\Omega^1_X)$
then the following assertions hold true.
\begin{enumerate}
\item If  $k \ge \ell$ then $\cal F$ is transversely affine.
\item If $k \ge \ell +1$ then $\cal F$ is defined by a closed logarithmic $1$-form.
\item If $k \ge  \ell+2$ then $\cal F$ is algebraically integrable.
\end{enumerate}
\end{theorem}
\begin{proof}
    Let $\Inv_0(\cal F) \otimes \bb C$ be the kernel of the restriction of the Chern class map to
    to $\Inv(\cal F)$.
    Basic Hodge theory associates to each element $D \in \Inv_0(\cal F)\otimes \mathbb C$
    a logarithmic $1$-form $\eta$
    with residue divisor equal to $D$. This association is not unique: $\eta$ is defined
    up to the addition of a global holomorphic $1$-form.

    The restriction of the logarithmic $1$-form $\eta$ to the leaves of $\cal F$ defines a holomorphic
    section of $\Omega^1_{\cal F}$, i.e. $\eta_{|\TF} \in H^0(X, \Omega^1_{\cal F})$.

    If the number $k$ of irreducible $\cal F$-invariant hypersurfaces is at least $\ell +1$ then we can produce a
    non-zero logarithmic $1$-form $\eta$ such that $\eta_{| \TF} = 0$. Hence $\eta$ is the sought logarithmic $1$-form
    defining $\cal F$. Item (2) follows.  If $k\ge \ell +2$ then we have two linearly independent logarithmic $1$-forms $\eta$ and $\eta'$
    which vanish along $\TF$. Since $\cal F$ is a codimension one foliation, there exists a non-constant rational function
    $f \in \mathbb C(X)$ such that $\eta = f \eta'$. Differentiation implies $df \wedge \eta' =0$. We conclude that $f$ is
    a rational first integral for $\cal F$. Item (3) follows.

    To prove item (1), consider the fiber of the Chern class map $c: \Inv_0(\cal F) \otimes \mathbb C \to H^2(X,\mathbb C)$ over
    the Chern class of the normal bundle. For each divisor $D$ such that $c(D) = c(\NF)$, we can construct a family, parameterized by $H^0(X,\Omega^1_X)$, of flat logarithmic connections
    on $\NF$ with residue divisor equal to $D$. If we restrict one such connection to a partial connection along the leaves of $\cal F$ then it will differ from Bott's
    partial connection by an element of $H^0(X, \Omega^1_{\cal F})$. Thus, if $k \ge \ell$ then at least one of such connections will have
    restriction to the leaves equal to Bott's partial connection. But this means exactly that $\cal F$ is transversely affine.
\end{proof}

\subsection{Examples}
The foliation $\mathcal H_5$ described in \cite[Theorem 2]{MP05} is a foliation on $\mathbb P^2$ with $\Omega^1_{\cal F} = \cal O_{\mathbb P^2}(4)$, $h^0(\mathbb P^2, \Omega^1_{\cal F})=15 = k-1$ invariant lines, and which is not transversely affine. This shows that item (1) is sharp.

The general Riccati foliation on $\mathbb P^1 \times \mathbb P^1$ with three invariant fibers (multiplicity one) and an invariant section
is a transversely affine foliation which is not transversely Euclidean. In this case, $\Omega^1_{\cal F} =  \cal O_{\mathbb P^1 \times \mathbb P^1}(1,0)$ and therefore
\[
  \dim NS(\mathbb P^1 \times \mathbb P^1 ) + h^0(\mathbb P^1 \times \mathbb P^1, \Omega^1_{\cal F}) - h^0(\mathbb P^1\times \mathbb P^1, \Omega^1_{\mathbb P^1\times \mathbb P^1})= 2 + 2 - 0= 4 \, .
\]
This shows that item (2) is sharp.

The general logarithmic $1$-form on $\mathbb P^2$
with poles along three distinct lines defines a foliation $\cal F$ with $\dim NS(\mathbb P^2) + h^0(\mathbb P^2, \Omega^1_{\cal F}) - h^0(\mathbb P^2, \Omega^1_{\mathbb P^2})= 2$,
three invariant lines, and which is not algebraically integrable. This shows that item (3) is sharp.

\subsection{Comparison}\label{S:comparison}
We have stated our version of Darboux-Jouanolou Theorem for projective manifolds, but it can be formulated for arbitrary compact complex
manifolds as in \cite{Ghys00}. For that, one has to take the integer $\ell$ equal to
\[
\dim \left\{ \coker \Div(X)\otimes \mathbb C \to H^1(X, \Omega^1_{X,cl}) \right\} + h^0(X, \Omega^1_{\cal F}) - h^0(X,\Omega^1_{X,cl})
\]
where $\Omega^1_{X,cl}$ stands for the $\mathbb C$-sheaf of closed holomorphic $1$-forms on $X$.

The bound presented in Theorem \ref{T:Jouanolou} item (3)  coincides with the bound
presented in \cite[Theorem 6.1]{Brunella00} and is slightly different from the bound
presented in \cite{Ghys00}, even if one specializes the latter to  projective manifolds.
The main difference comes from the use here of $H^0(X,\Omega^1_{\cal F})$ instead of
$H^0(X,\Omega^2_X\otimes \NF)$. To relate the dimensions of these two finite dimensional vector spaces observe
that the sheaf $\Omega^1_{\cal F}$ injects into $\Omega^2_X \otimes \NF$. Indeed, if $\omega \in H^0(X, \Omega^1_X \otimes \NF)$ is a twisted $1$-form defining $\cal F$ then
for any local section $\eta \in \Omega^1_{\cal F}(U)$ ($U$ sufficiently small) we can unambiguously define  $\omega \wedge \eta$ by considering a lift of $\eta_{|U - \text{sing}(\mathcal{F})}$ to
$\Omega^1_{X}(U- \text{sing}(\mathcal{F}))$,
extending the result using Hartog's, and taking the wedge product with $\omega$. The final result  is clearly independent of the choice of the lift.

\subsection{Quasi-invariant divisors} We will say that an irreducible hypersurface $H$ is quasi-invariant by a codimension one foliation
$\cal F$ if $H$ is not  $\cal F$ invariant and the restriction of $\cal F$ to $H$ is algebraically integrable. The subgroup of $\Div(X)$ generated
by the irreducible hypersurfaces quasi-invariant by $\cal F$ will be denoted by $\QInv(\cal F)$.

For a foliation $\cal F$  on a surface $S$ the concept is not very useful. According to our definition, any divisor in
 $\Div(S)$ is a member of $\Inv(\cal F) + \QInv(\cal F)$.
Anyway, these foliations provide a natural source of examples of quasi-invariant divisors.

\begin{proposition}
Let $X$ be  a projective manifold  of dimension at least three which   admits a dominant rational
map $p : X \dashrightarrow S$ to a smooth projective surface
$S$. If $\cal G$ is a foliation on $S$ and we set $\cal F = p^* \cal G$ then $p^* \Div(S) \subset \Inv(\cal F) + \QInv(\cal F)$.
\end{proposition}
\begin{proof} Let $C$ be an irreducible curve on $S$. The pull-back of $C$ consists in a finite number
of irreducible hypersurfaces. The ones which dominate $C$ are either $\cal F$-invariant (when $C$ is $\cal G$ invariant)  or are not $\cal F$-invariant but
foliated by fibers of $p$. Hence the irreducible hypersurfaces dominating $C$ are either $\cal F$-invariant or quasi-invariant as wanted.
If a hypersurface is contracted to a point then  after blowing-up sufficiently many
times its image, we reduce to the previous situation.
\end{proof}

\section{Flat divisors}
The proof of Theorem \ref{THM:quasiJouanolou} will be presented in Section \ref{S:proof} and is based on
properties of flat $S^1$-divisors (in the sense of \cite{Pereira06}) and their associated foliations.
In this section we recall several results from \cite[\S 2, \S3]{Pereira06} relating to $S^1$-flat divisors, and
adapt/extend some of them to our current goals.

\subsection{Definition}
Let $X$ be a complex manifold.
A divisor $D$ is called $S^1$-flat if $\cal O_X(D)$ is in the
image of the natural morphism
\[
    H^1(X, S^1) \rightarrow H^1(X, \cal O_X^*) = \Pic(X)
\]
induced by the inclusion $S^1 \to \cal O_X^*$, where $S^1$ denotes the constant sheaf of complex numbers of modulus $1$.

\begin{lemma}
Let $X$ be a compact K\"ahler manifold.  A line bundle $L$ on $X$
is  $S^1$-flat if and only if $c(L) = 0 \in H^2(X,\mathbb C)$.
\end{lemma}

\subsection{Associated $1$-forms}\label{SS:forms}
If $D \neq 0$ is an $S^1$-flat divisor then there exists a
natural construction of a closed logarithmic $1$-form $\omega_D \in H^0(X, \Omega^1(\log D))$
with purely imaginary periods.  It goes as  follows. Choose a sufficiently fine open
covering $U_i$ of $X$ and meromorphic functions $f_i:U_i \dashrightarrow \bb P^1$ such
that $D\vert_{U_i} = (f_i)_0 - (f_i)_\infty$ and $f_i = t_{ij} f_j$ where $t_{ij}$ is
constant of modulus $1$. Define $\omega_D$ locally as the
logarithmic differential of $f_i$, i.e.
$$
    {\omega_D}\vert_{U_i} = \frac{df_i}{f_i}.
$$
If $X$ is compact then the $1$-form $\omega_D$ is defined unambiguously. In general, $\omega_D$ is well-defined up to the addition of an exact holomorphic differential.

Let $D \neq 0$ be a $S^1$-flat divisor on a complex manifold
and let $\omega_D$ be a logarithimic $1$-form associated to $D$ as in
\S \ref{SS:forms}.  Since $d\omega_D = 0$
this defines a codimension one foliation denoted
$\cal F_D$.

We record some facts about the foliation $\cal F_D$.

\begin{lemma}\label{importantfacts}
Let $D\neq 0$ be a $S^1$-flat divisor and let $\omega_D$ and $\cal F_D$ be as above.
Then the following assertions hold true.
\begin{enumerate}
\item The support of $D$  is $\cal F_D$-invariant.
\item The periods of $\omega_D$ are purely imaginary and so $F = \vert e^{\int \omega_D} \vert: X \rightarrow [0, \infty]$
is a well-defined non-constant continuous function.
\item The function $F$ provides a real first integral of $\cal F_D$, i.e.,
$F$ is constant on leaves of $\cal F_D$.
\end{enumerate}
\end{lemma}

\subsection{Closure of leaves}
We start this section  with a simple observation.

\begin{lemma}\label{L:uncountable}
Let $\cal F_1$ and $\cal F_2$ be two codimension one foliations on a connected
complex manifold (not
necessarily compact) $X$.
Suppose that $\cal F_1$ and $\cal F_2$ have an uncountable number
of distinct leaves in common.  Then $\cal F_1 = \cal F_2$.
\end{lemma}
\begin{proof}
    Let $L_\lambda$ be the set of common leaves of $\cal F_1$ and $\cal F_2$.
    By a simple counting argument we know that there must be some point $p \in X$ and an open set $p \in U \subset X$
    such that the $L_\lambda$ accumulate to some point in the interior of $U$.

    The tangency locus of $\cal F_1$ and $\cal F_2$ in $U$ contains  the smallest analytic set containing all the leaves $L_\lambda$, which is all of $U$.
    Thus $\cal F_1$ and $\cal F_2$ agree on a non-empty open subset and so must be equal on all of $X$.
\end{proof}

Proposition \ref{P:noncompactversion} below is a version of \cite[Proposition 5.2]{Pereira06} in the non-compact case.
The argument is essentially the same, but for convenience we include it here.

\begin{proposition}\label{P:noncompactversion}
Let $X$ be a complex manifold (not necessarily compact)
and let $D$ be an  $S^1$-flat divisor on $X$.
Let $\cal F$ be a codimension one foliation on $X$
and suppose $\cal F$ admits a leaf $L$ such
that
\begin{enumerate}
     \item the smallest complex analytic subset of $X$ containing $L$ is $X$ itself; and
     \item the topological closure $\overline L$ of $L$ is disjoint from the support of $D$, i.e. $\overline{L}\cap \text{supp}(D) = \emptyset$.
\end{enumerate}
Then the foliations $\cal F$ and $\cal F_D$ coincide.
\end{proposition}
\begin{proof}
Without loss of generality by \cite[Lemma 5.1]{Pereira06} we may assume that $D = D_+ - D_-$ where $D_+, D_- \geq 0$
and $D_+ \cap D_- = \emptyset$.

The foliation $\cal F_D$ is defined by a closed logarithmic $1$-form $\omega_D$ with a polar set supported on $D$ and by Lemma \ref{importantfacts} admits a real first integral $F:X \rightarrow [0, \infty]$ such that $F^{-1}(0) = \text{supp}(D_+)$ and $F^{-1}(\infty) = \text{supp}(D_-)$. Consider the restriction of $F$ to $L$.  Since $F$ is the modulus of a holomorphic function we have two cases: $F\vert_L$ is either (i) constant or is (ii) open.

In case (i) $F\vert_L$ is constant,  $L$ is a leaf of both $\cal F$ and $\cal F_D$. So, the tangency locus of $\cal F$ and $\cal F_D$ must contain the smallest complex analytic subset of $X$ containing $L$.  This, however, is $X$ itself and hence $\cal F = \cal F_D$.

Suppose that we are in case (ii), i.e., $F\vert_L$ is an open map.
Since $F\vert_L$ is open and $L$ is bounded away from $D$ we know that there exists positive real numbers $m < M$ such that  $F(L)$ is the interval $(m, M) \subset \mathbb R$.

Recall that $F$ can be locally written as
$$F = \vert e^{\int \omega_D} \vert$$
and that all the periods of $\omega_D$ are purely imaginary complex numbers. Let $f$ denote the multivalued function $e^{\int \omega_D}$
and let $\pi:\widetilde{X} \rightarrow X$
be the cover associated to $f$.  Write $\widetilde{L} = \pi^{-1}(L)$
and let $\widetilde{\cal F}$ be the pulled back foliation.
If $\tilde{f}$ is the lift of $f$ we have that $$\tilde{f}:\tilde{X} \rightarrow \bb P^1$$ is holomorphic.

We know that $\tilde{f}(\widetilde{L}) \subset \bb C^*$ is open and relatively compact. Thus  $\partial(\tilde{f}(\widetilde{L}))$ is (uncountably) infinite. Pick $p \in \pi^{-1}(\partial L)\cap \tilde{f}^{-1}(\partial(\tilde{f}(\widetilde{L})))$
and let $\tilde{L}_p$ be the leaf of $\widetilde{\cal F}$ through $p$.
By the maximum principle we see that $\tilde{f}$ must be constant on $\tilde{L}_p$,
and therefore $\pi(\tilde{L}_p)$ is a leaf of both $\cal F$ and $\cal F_D$.

Since the monodromy group associated to $f$ is countably generated, this implies
that $\cal F$ and $\cal F_D$ have uncountably many leaves in common. Lemma \ref{L:uncountable} implies that $\cal F$ and $\cal F_D$
are equal.
\end{proof}

\subsection{Subfoliations defined by almost invariant divisors}

The following Proposition is the  main technical  point in the proof of Theorem \ref{T:quasiJouanolou}/Theorem \ref{THM:quasiJouanolou}.

\begin{proposition}\label{P:maintechnicalstep}
Let $D_1$ and $D_2$ be two $S^1$-flat
divisors on a compact complex threefold $X$
and let $L\xrightarrow{j}X$ be an immersion of a smooth complex surface.
Suppose
\begin{enumerate}
    \item\label{I:1} the connected components of $j^{-1}(\text{supp}(D_i))$ are compact for $i=1,2$; and
    \item\label{I:2} there exists some connected component $E$ of $j^{-1}(\text{supp}(D_1))$ such that
        $E\cap j^{-1}(\text{supp}(D_2)) = \emptyset$; and
    \item\label{I:3} if we write $D_1 = D_{1, +} - D_{1, -}$ where $D_{1, +}, D_{1, -} \geq 0$ are effective divisors,
        then $D_{1, +}\cap D_{1, -}\cap L  = \emptyset$.
\end{enumerate}
Let $\cal G$ be the restriction of $\cal F_{D_1}$ to $L$. Then there exists a neighborhood of $E$, saturated by $\cal G$, and filled up with
$\cal G$-invariant compact curves.
\end{proposition}
\begin{proof}
Let $E$ be as in assumption (\ref{I:2}). According to assumption (\ref{I:1}),  $E$ is a compact $\cal G$-invariant curve.

Let $F$ be the real  first integral of $\cal F_{D_1}$ given by Lemma \ref{importantfacts},
which by restriction to $L$ gives a real  first integral of $\cal G$.  Assumption (\ref{I:3}) guarantees the existence
of an open neighborhood $U$ of $E$ in $L$ which is saturated by $\cal G$, i.e. any leaf of $\cal F_{D_1}$ which
meets $U$ is in fact contained in $U$.
Moreover, shrinking $U$ even further we
may assume that $U\cap j^{-1}(D_2) = \emptyset$.

For $p \in U$, let $\ell_p$ denote the leaf of $\cal G$ through $p$
and let $\overline{\ell_p}$ denote the topological closure
of $\ell_p$ in $U$.

If $\overline{\ell_p} = \ell_p$ (up to the addition of finitely many singular points of $\mathcal G$)  for all $p$ then we are done.
Otherwise, there is some $p$ such that the cardinality of $\overline{\ell_p} - \ell_p$ is infinite.
Since $\overline{\ell_p} \cap j^{-1}(D_2) = \emptyset$ we may apply Proposition \ref{P:noncompactversion}
to conclude that $\cal G = \cal H$ where $\cal H$ is the restriction
of $\cal F_{D_2}$ to $L$.

If $\omega_i$ is the closed logarithmic $1$-form defining $\cal F_{D_i}$ we have
that $\omega_1 = H \omega_2$ for some meromorphic function $H$ on $L$. Assumption (\ref{I:2}) guarantees
the non-constancy of $H$. Taking $d$ of both sides shows that $dH \wedge \omega_2 = 0$ and
so $\cal G$ has a meromorphic first integral.
However, since $D_{i, +}\cap D_{i, -} \cap L = \emptyset$ we see that this first
integral may be taken to be holomorphic. The restriction of the level sets of $H$ to
$U$ fills up the neighbourhood with $\cal G$-invariant compact curves.
\end{proof}

\section{Proof of Theorem \ref{THM:quasiJouanolou}}\label{S:proof}

\subsection{General complete intersection subvarities}


We recall a definition from \cite{AD14} which is based on \cite[Section 2.3]{LPT11}.

\begin{defn}
\label{D:algebraicpart}
Let $\cal F$ be a holomorphic foliation on a normal variety $X$. Then there exists
a normal variety $Y$ (unique up to birational equivalence),
a dominant rational map with connected fibres $\phi: X \dashrightarrow Y$ and a holomorphic
foliation $\cal G$ on $Y$ such that the following holds
\begin{enumerate}
\item $\cal G$ is purely transcendental, i.e., through a general point of $Y$ there is no positive
dimensional variety tangent to $\cal G$, and

\item $\cal F = \phi^*\cal G$.
\end{enumerate}
The foliation induced by $\phi$ is called the {\bf algebraic part} of $\cal F$
and is denoted $\cal F^{alg}$.
\end{defn}

For a projective manifold $X$, we say that a subvariety $V \subset X$ is general complete intersection
if $V$ is obtained by embedding $X$ in a projective space by means of a complete linear system defined by sufficiently high multiple of an ample line bundle
and intersecting the result with sufficiently general hyperplane sections. In particular, if $V$ is a general complete intersection in
$X$ of  dimension at least three then Lefschetz-type results \cite[Exposé XII, Corollaire 3.6]{MR2171939} guarantee that $V$ and $X$ have isomorphic
Picard groups.

\begin{lemma}
Let $\cal F$ be a codimension one foliation on a projective variety $X$
and let $\phi:X \dashrightarrow Y$ be the rational map corresponding
to $\cal F^{alg}$. Let $V \subset X$ be a general complete intersection variety
and let $\cal F_{|V}$ be the restricted foliation.
Then $(\cal F_{|V})^{alg}$ is induced by $\phi\vert_D:D \dashrightarrow \phi(D) \subset Y$.
\end{lemma}
\begin{proof}
Let $W$ be the normalization of $\phi(D)$ and let $\cal G_{|W}$ be the restricted
foliation. Since $V$ is a general complete intersection variety we see
that $\cal G_{|W}$ is purely transcendental.

Thus the map $\phi\vert_D:V \dashrightarrow W$ satisfies both conditions
of Definition \ref{D:algebraicpart}.
\end{proof}

\subsection{Reduction of singularities}\label{S:Cano}
Recall from \cite[page 910]{Cano} that a foliation $\cal F$ has dicritical singularities if there exists
a sequence of blow-ups with $\cal F$ invariant centers such that an
irreducible component of the exceptional divisor is not invariant by the resulting
foliation.

If $\cal F$ is a codimension one foliation on  a threefold $X$ then
\cite{Cano} shows the existence of a  sequence of blowups in foliation invariant centres
$$\pi:(Y, \cal G) \rightarrow (X, \cal F)$$
so that $\cal G$ does not have dicritical singularities.  
Such a resolution is not known in higher dimensions.

If a foliation $\cal F$ does not have dicritical singularities
then the group $\Inv(\cal F)  + \QInv(\cal F)$ behaves rather well
with respect to birational morphisms.

\begin{lemma}\label{L:Inv+QInv}
Let $\cal F$ be a foliation on a projective manifold $X$ of dimension three. Let $\pi : Y \to X$ be a birational
morphism from another projective manifold $Y$ to $X$. If $\cal F$ does not have dicritical singularities
then every fibre of $\pi$ is tangent to $\pi^*\cal F$.  In particular,
$\pi^* (\Inv(\cal F) + \QInv(\cal F)) \subset \Inv(\pi^* \cal F) + \QInv(\pi^* \cal F)$.
\end{lemma}
\begin{proof}
To prove our first claim suppose for sake of contradiction that there exists some $P \in X$ and a component of $\pi^{-1}(P)$
which is not tangent to the foliation.  Perhaps blowing up more we may assume that this component
is a non-invariant divisor $E$. Observe that $E$ being non-invariant implies
$P \in \text{sing}(\cal F)$.  Let $\pi_0:Y_0 \rightarrow X$ be the blow up at $P$
and let $Z_0$ be the centre of $E$ on $Y_0$.  By \cite[Th\'eorem\`e 4]{CM92} $\text{exc}(\pi_0)$ is $\pi_0^*\cal F$ invariant
and so $Z_0 \subset \text{sing}(\pi_0^*\cal F)$.  Blowing up in $Z_0$ and continuing inductively
we get a sequence of blow ups extracting divisors invariant under the foliation.  However, by \cite[Lemma 2.45]{KM98}
this process will eventually extract $E$, a contradiction.

To prove our second claim it suffices to show that irreducible components of the exceptional divisor of $\pi$
are either invariant or quasi-invariant.  Let $D$ be one such component.  If $D$ is non-invariant the
fibres of $D \rightarrow \pi(D)$ are tangent to $(\pi^*\cal F)_{|D}$, in particular
$(\pi^*\cal F)_{|D}$ is algebraically integrable and so $D$ is quasi-invariant.
\end{proof}

Another advantage of non-dicritical singularities can be seen in the following proposition.

\begin{proposition}\label{P:nondicritical is good}
Let $X$ be a smooth threefold and let $\cal F$ be a codimension one foliation with non-dicritical
singularities. Let $C \subset X$ be a projective curve tangent to $\cal F$. Then there exists
a germ of a $\cal F$-invariant surface $S$ containing $C$.
\end{proposition}
\begin{proof}
This follows from \cite[Theorem IV.1.1]{CC92}.
\end{proof}

Proposition \ref{P:nondicritical is good}  allows  to extend
immersions $j^{\circ} : L^{\circ} \to X - \mathrm{sing}(\cal F)$ of leaves of $\cal F_{|X - \mathrm{sing}(\cal F)}$ to
holomorphic maps $j: L \to X$ with the property that the pre-image of quasi-invariant divisors have compact connect components.

Notice that the non-dicritical assumption is necessary in Proposition \ref{P:nondicritical is good}  as the following example by
\cite{Jouanolou79} shows.

\begin{example}
The dicritical foliation defined on $\bb C^3$ by
\[
\omega = (x^my-z^{m+1})dx+(y^mz-x^{m+1})dy +(z^mx - y^{m+1})dz
\]
has no separatrices at the origin.  However, every line passing
through the origin is tangent to the foliation.
\end{example}

\subsection{Strenghtening of Theorem \ref{THM:quasiJouanolou}}\label{S:statement}
At this point we have all the concepts necessary to state the more precise version of Theorem \ref{THM:quasiJouanolou} mentioned in the Introduction.

\begin{theorem}\label{T:quasiJouanolou}
Let $\cal F$ be a codimension one foliation on a projective manifold $X$. If the singularities of  the restriction of $\cal F$ to a
three dimensional general complete intersection subvariety of $X$ are non-dicritical and
\[
    \dim  \QInv(\cal F)\otimes \bb C  \ge \dim \frac{NS(X)\otimes \mathbb C}{c(\Inv(\cal F)\otimes \bb C)} + 2
\]
then there exists a projective surface $S$, a dominant rational map $p: X \dashrightarrow S$, and a foliation $\cal G$ on $S$
such that $\cal F = p^* \cal G$.
\end{theorem}

\subsection{Reduction to dimension three}

\begin{lemma}
\label{L:reductiontodim3}
Let $X$ be a projective manifold of dimension $n \geq 3$ and let $\cal F$ be a codimension one foliation
on $X$.  Let $V \subset X$ be a three dimensional general complete intersection and let $\cal F_{|V}$ be the restricted foliation.
Then
\begin{enumerate}
\item $\cal F_{|V}$ is algebraically integrable if and only if $\cal F$ is algebraically integrable and

\item $\cal F_{|V}$ is pulled back from a rational map to a surface if and only if $\cal F$ is pulled back from
a rational map to a surface.
\end{enumerate}
Thus, without loss of generality it suffices to prove Theorem \ref{T:quasiJouanolou} and Theorem \ref{THM:quasiJouanolou}
in the case where $X$ is a threefold.
\end{lemma}
\begin{proof}
The if direction  of items (1) and (2) are clear.

To prove the other direction in both items,  let
$\phi:X \dashrightarrow Y$ be the fibration corresponding to $\cal F^{alg}$.
On one hand we know that $\phi\vert_V:V \dashrightarrow Y$ corresponds
to $(\cal F_{|V})^{alg}$ and so by assumption $\dim \phi(V) = 1$ (respectively $ =2$).
But since $V$ is a general complete intersection of dimension $3$ this forces $\dim Y =1$
(respectively $ = 2$).

To see our final claim, notice that $D$ is a $\cal F$-(quasi-)invariant divisor
if and only if $D\vert_V$ is $\cal G$-(quasi-)invariant.
This gives us
\[
\dim  \QInv(\cal F_V)\otimes \bb C = \dim  \QInv(\cal F)\otimes \bb C \, .
\]
Thus the assumptions of Theorem \ref{THM:quasiJouanolou} are also satisfied by $\cal F_{|V}$. Hence, to prove Theorem \ref{THM:quasiJouanolou}, it suffices to do it in dimension three.

We will now  verify the analogue claim for Theorem \ref{T:quasiJouanolou}. As we are assuming that $V$ is a general complete intersection of dimension three, Lefschetz-type results
guarantee that $NS(V) = NS(X)$. Consequently,
\[
\dim \frac{NS(V)\otimes \mathbb C}{c(\Inv(\cal F_V)\otimes \bb C)} =
\dim \frac{NS(X)\otimes \mathbb C}{c(\Inv(\cal F)\otimes \bb C)} .
\]

Thus, assuming Theorem \ref{T:quasiJouanolou} for threefolds,
\[
    \dim  \QInv(\cal F)\otimes \mathbb C  \ge \dim \frac{NS(X)\otimes \mathbb C}{c(\Inv(\cal F)\otimes C)} + 2
\]
implies that $\cal F_{|V}$ is either algebraically integrable or pulled back from a surface, in which
case we conclude the same is true of $\cal F$.
\end{proof}

\subsection{From non-trivial deformations to pull-backs}

\begin{lemma} \label{algebraicity}
Let $X$ be a smooth projective threefold
and let $\cal F$ be a codimension one foliation on $X$.
Let $C$ be a curve (not necessarily irreducible) each component
of which is tangent to $\cal F$.
Suppose that $\cal F$ is smooth near $C$, i.e. $C$ does not intersect the singular set of
$\cal F$.
Let $L$ be the germ of  leaf containing $C$.  Suppose that $C$
has a non-trivial deformation in $L$.
Then either
\begin{enumerate}
\item $L$ is algebraic (i.e., its Zariski closure is a surface); or
\item there is a foliation in algebraic curves tangent to $\cal F$.
\end{enumerate}
\end{lemma}
\begin{proof}
Let $W$ be the closed subscheme
of the Hilbert scheme parametrizing subvarieties tangent to $\cal F$ and
let $W_0$ be the connected component of $W$ containing $[C]$.
Let $\gamma:\Delta \rightarrow W_0$ be the map associated to
the deformation and let $B$ be the Zariski closure
of $\gamma(\Delta)$.
Let $\pi:U \rightarrow W_0$ be the universal family
with evaluation map $e:U \rightarrow X$.
Consider the restriction
$e:U_B = U\times_{W_0}B \rightarrow X$.

If $e:U_B \rightarrow X$ is dominant then if $S \subset B$
is a general hypersurface we have that $S$ parametrizes
a two dimensional family of curves tangent to $\cal F$ and dominating
$X$ which gives our desired subfoliation by curves.

Otherwise, $e:U_B \rightarrow X$ dominates some divisor $D$.
Let $p_t = \gamma(t)$ and let $\Gamma = \gamma(\Delta)$.
One one hand, we have $e(\pi^{-1}(p_t)) \subset D$, on the other, if
$V$ is some small neighborhood
of $C$ we see that $L\cap V \subset e(\pi^{-1}(\Gamma))$.
Hence, the Zariski closure of $L$ is contained in $D$.
\end{proof}

\subsection{Algebraicity criterion}

\begin{lemma}\label{curvelemma}
Let $L$ be a  leaf of a foliation on a projective $3$-fold $X$ containing
an effective and  nef divisor $\Sigma$ with compact and connected support. If $\Sigma^2>0$ then $L$ is algebraic.
\end{lemma}
\begin{proof}
Write $\Sigma = \Sigma_1 +\Sigma_0$ where $\Sigma\cdot C>0$ for all $C$ in the support of $\Sigma_1$
and $\Sigma\cdot C = 0$ for all $C$ in the support of $\Sigma_0$.
We define
\[\Sigma^{(1)} = (1+\epsilon)\Sigma_1+\Sigma_0\]
for some sufficiently small choice of $\epsilon >0$
and we define $\Sigma^{(1)}_1$ and $\Sigma^{(1)}_0$ in a similar manner.
For $\epsilon>0$ small we see that $\supp{\Sigma_1} \subset \supp{\Sigma_1^{(1)}}$
but also, if $C \subset \supp{\Sigma_0}$ and $C\cap\supp{\Sigma_1} \neq \emptyset$ then $C \subset \supp{\Sigma_1^{(1)}}$.
Continuing inductively, since $\Sigma$ is connected,  we eventually produce a divisor $\Sigma^{(m)}$ such
that $\Sigma^{(m)}_1 = \Sigma^{(m)}$.  For $N \in \bb N$ sufficiently divisble
\[\sigma = N\Sigma^{(m)}\]
is a Cartier divisor.
By construction $\sigma \cdot C >0$ for all irreducible $C \subset \supp{\sigma}$.
In particular, we see that $N_{\sigma/L}$ is ample.

Thus by \cite[Theorem 6.7]{Hartshorne68} the Zariski closure
of $L$ in $X$ is a projective surface.
\end{proof}

\subsection{Proof of Theorem \ref{T:quasiJouanolou}}
The assumption on the number of quasi-invariant hypersurfaces  allows
the construction of  two  $S^1$-flat divisors in $\Inv(\cal F) + \QInv(\cal F)$, say $D_1$ and $D_2$,
which have distinct quasi-invariant hypersurfaces in their respective supports.

We may assume, by Lemma \ref{L:reductiontodim3}, that $\cal F$ is a non-dicritical foliation on
a projective $3$-fold.

Let $D \in \Inv(\cal F) + \QInv(\cal F)$ and let $\pi:Y \to X$ be a birational morphism
from a projective manifold $Y$  to $X$.  As we are assuming that $\cal F$ has non-dicritical singularities,
then $\pi^* D \in \Inv(\pi^* \cal F) + \QInv(\pi^* \cal F)$ according to Lemma \ref{L:Inv+QInv}. Therefore, we may freely replace
$X$ and $\cal F$ by any resolution without interfering with the  existence of divisors $D_1$ and $D_2$ as above.

We may write $D_i = D_{i, +} - D_{i, -}$ where $D_{i, +}, D_{i, -} \geq 0$ are effective divisors.
Let $B_i$ be the scheme theoretic intersection of $D_{i, +}$ and $D_{i, -}$ and let $\pi:Y \rightarrow X$ be the blow up
in $B_1$ and $B_2$ followed by a resolution of singularities.
Let $p \in D_{i, +} \cap D_{i, -}$ and let $U$ be an open affine containing $p$ such that $D_{i, +}\cap U = (f_{i, +} = 0)$
and $D_{i, -}\cap U = (f_{i, -} = 0)$. This gives us a rational map $F = [f_{i, +}: f_{i, -}]:U \dashrightarrow \bb P^1$
such that $F^*0 - F^*\infty  = D_i\cap U$.
Observe that $\pi$ resolves the indeterminacy locus of $F$ and so if we write
$\pi^*D_i = A_i - B_i$ where $A_i, B_i \geq 0$ are effective then $A_i$ and $B_i$ have disjoint support.
Thus, replacing $X$, $\cal F$ and $D_i$ by $Y$, $\pi^*\cal F$ and $\pi^*D_i$
we may assume that $D_{i, +} \cap D_{i, -} = \emptyset$.

Let $H_1$ be a quasi-invariant invariant hypersurface contained in $\supp {D_1}$ and let 
$j:L \rightarrow X$ be a general leaf of $\cal F$ passing through a sufficiently general point 
of $H_1$. Let $E$ be a  connected component of the divisor $j^*(D_{1})$ containing a connected component
of $j^* H_1$. 

As $D_{1,+}$ and $D_{1,-}$ are disjoint effective divisors with the same Chern classes, we have that
$E$ or $-E$ is an effective divisor in $L$ with  zero self-intersection, i.e. $E^2=0$. Perhaps after replacing $E$ by $-E$, we 
can assume that $E$ is effective. Note also that any  curve $C$ in the support of $E$ satisfies $E \cdot C=0$. 
Hence $E$ is effective and  nef. 

Consider now the restriction of $D_2$ to $L$, i.e. $j^*D_2$. Our choice of $L$ (passing through a general point of $H_1 \varsubsetneq \supp(D_2)$) 
guarantees that the support of $j^*D_2$ does not contain the support of $E$. If the support of $j^*D_2$  intersects $E$, then we can pick a
curve $F$ in  $\supp(j^*D_2)$ such that $E \cdot F>0$. For $k\gg0$, the divisor $(kE + F)^2>0$ is nef, 
effective and has positive self-intersection. Lemma \ref{curvelemma} implies $L$ is algebraic. Since $L$ is general,
we deduce that $\cal F$ is algebraically integrable.

If instead the support of $j^*D_2$ does not intersect the support of
$E$ then the hypotheses of Proposition \ref{P:maintechnicalstep}  are satisfied. So,
 if $\cal G$ is the restriction of the  foliation $\cal F_{D_1}$
to $L$ then $\cal G$ has a saturated neighborhood
of $E$ filled up with invariant compact curves.

Since $c_1(j^*D_i) = 0$ this implies that in fact $E$ is a (multiple of a)
fibre of $f$ and so (a multiple of) $E$ must move in $L$.  Since $\cal F$ is smooth
near $E$ we may apply Lemma \ref{algebraicity} to conclude.
\qed

\subsection{Proof of Theorem \ref{THM:quasiJouanolou}}
We may assume, by Lemma \ref{L:reductiontodim3}, that $\dim X = 3$. Moreover,
it suffices to verify the conclusion on some resolution of $X$ and $\cal F$. So passing to a resolution
we may assume that $\cal F$ has simple singularities, in particular, they are non-dicritical, cf. Section \ref{S:Cano}. We can apply Theorem \ref{T:quasiJouanolou} to conclude.
\qed

\subsection{Remarks on the assumptions of Theorem \ref{THM:quasiJouanolou}}
It is possible that adaptations analogous to the ones described in Section
\ref{S:comparison}, might lead to a version of our result for arbitrary complex
manifolds. One obstruction to carry out our argument in general is the lack
of reduction of singularities for foliation on manifolds of dimension strictly greater
than $3$. In the projective case, the lack of such result is bypassed by considering the
restriction of the foliation to the general $3$-dimensional submanifold. Another difficulty
comes from the lack of properness for irreducible components of the spaces parametrizing the codimension two subvarieties
invariant by the foliation $\mathcal G$, used in the proof Lemma \ref{algebraicity}.

\section{Quasi-invariant divisors  in positive characteristic}

\subsection{Basic concepts}\label{S:charp review}
We recall some basic facts and definitions about foliations in characteristic $p$.

Let $X$ be a smooth variety over an algebraically closed field of characteristic $p$. For us, a foliation $\cal F$ on
$X$ is, as in characteristic zero, given by a saturated subsheaf $\TF$ of $T_X$ which is closed under
the Lie bracket.

Let $v \in T_X(U)$ be some local section
of the tangent sheaf.  An elementary computation shows that the $p$-th power of $v$
is still a vector field on $U$. Taking $p$-th powers of vector fields gives an $\cal O_X$-linear map
\[
F: \Frob^*\TF \rightarrow T_X/\TF \, ,
\]
where $\Frob : X \to X$ is the (absolute) Frobenius morphism.

\begin{defn}
Let $X$ be a normal variety over a field of characteristic $p$ and let $\cal F$ be a foliation
on $X$.  We say that $\cal F$ is   $p$-closed  provided $F: \Frob^*\TF \rightarrow T_X/\TF$
is the zero morphism.
\end{defn}

It is easy to see that $p$-closedness is a birational invariant and that if $\cal F$
is algebraically integrable then it is $p$-closed.

Unlike in characteristic zero, the Frobenius Theorem does not hold in positive characteristic.
The structure of the foliation depends on whether it is $p$-closed or not.
If a foliation is not $p$-closed then for a sufficiently general  closed point $x \in X$ there is no (formal) $\cal F$ invariant subvariety
through $x$. If instead the foliation is $p$-closed then through any closed point $x \in X$ there exists
infinitely many algebraic $\cal F$ invariant subvarieties  through $x$. For more details and references on the subject
we redirect the reader to \cite[Lecture III]{MR1468476} and \cite[Section 7]{LPT11}.

The dichotomy on the behaviour of $\cal F$ describe in the previous paragraph suggests the following adaptation of the
notion  of quasi-invariance to  positive characteristic.

\begin{defn}
\label{d:charpqinv1}
Let $X$ be a normal variety over an algebraically closed field of characteristic $p>0$ and let $\cal F$ be a codimension one foliation
on $X$. We say that a hypersurface $H$ is  quasi-invariant if it is not invariant and the foliation
restricted to $H$ is $p$-closed.
\end{defn}


\subsection{Proof of Theorem \ref{THM:charp}}
Denote by $D_1, D_2, ...$ the collection of quasi-invariant hypersurfaces.
Suppose that $\cal F$ is not $p$-closed, then the map $\Frob^*\TF \rightarrow T_X/\TF$
given by $v \mapsto v^p$ is generically surjective, and so the kernel $K$
is a rank $n-2$ reflexive subsheaf of $\Frob^*\cal F$ where $\text{dim}(X) = n$.
We claim that $K = \Frob^*\TG$ for a $p$-closed foliation $\cal G$ on $X$.

By \cite[Proposition 6.1]{PT13} there exists a saturated subsheaf $\TG\subset T_X$
such that $K = \Frob^*\TG$.

We first prove that $\cal G$ is a $p$-closed distribution.
It suffices check this locally in a Zariski neighborhood of a sufficiently general point $Q \in X$. In particular, we can assume
that each $D_i$ is cut out by a single equation $f_i = 0$ and that $\cal F$ is defined by a regular $1$-form
$\omega$.

After passing to a smaller Zariski neighborhood of $Q$ according to \cite[Lemma 6.1]{CLNLPT}
we may find regular vector fields $v_{i, 1}, ..., v_{i, n-1}$ such that
$[v_{i, j}, v_{i, k}]=0$ and
$v_{i, 1}, ..., v_{i, n-1}$ generate $\TF$ where $n = \dim(X)$.
Furthermore, we may choose $v_{i, 1}, ..., v_{i, n-2}$ so that
$v_{i, 1}, ..., v_{i, n-2}$ leave $D_i$ invariant and
so are tangent to $\cal F_{|D_i}$, the restriction of the foliation $\cal F$  to $D_i$.
We have then for $1 \leq j \leq n-2$ that $v_{i,j}^p$ is still tangent to $\cal F_{|D_i}$ and so
$\omega(v_{i, j}^p) \in (f_i)$.
We have two cases
\begin{enumerate}
\item $\omega(v_{i, n-1}^p) \in (f_i)$ and
\item $\omega(v_{i, n-1}^p) \notin (f_i)$.
\end{enumerate}

In case (1) we see that the map $\Frob^*\TF \rightarrow T_X/\TF$ vanishes along $D_i$,
and so if case (1) holds for infinitely many $D_i$ we have that $\TF$ is $p$-closed.

Thus, for all but finitely many $i$ we are in case (2).
Suppose that $\cal G$ is generated by vector fields $w_1, ..., w_{n-2}$. Then for all $i$ we may write
\[
w_k = \sum_{j = 1}^{n-1}a^k_{i, j}v_{i, j}
\]
with $a^k_{i, j} \in \cal O_{X, Q}$.

Observe that
\[
w_k^p = \sum (a^k_{i, j})^pv^p_{i, j} \mod \TF.
\]
Since $\omega(w_k^p) = 0$ for all $k$, $\omega(v^p_{i, j}) \in (f_i)$ for $1 \leq j \leq n-2$ and
$\omega(v^p_{i, n-1}) \notin (f_i)$, applying
$\omega$ to both sides gives us that
$(a^k_{i, n-1})^p \in (f_i)$, hence $a^k_{i, n-1} \in (f_i)$.
But this implies that for all $k,$ $w_k$ (and hence $w_k^p$) leaves $D_i$ invariant,
and so $\cal G$ leaves $D_i$ invariant. Since $w_k^p$
is tangent to $\cal F$ and leaves $D_i$ invariant,
we see that it is tangent to $\cal G$ along $D_i$.
Thus the tangency locus of $\cal G$ and $w_k^p$ consists of infinitely many
divisors and so $\cal G$ is $p$-closed.

We now claim that $\cal G$ is closed under Lie bracket.
Since $\cal G$ leaves $D_i$ invariant for all $i$ we see that $[w_j, w_k]$ leaves $D_i$ invariant
for all $i$. Thus $[w_j, w_k]$ is tangent to $\cal F$ and leaves $D_i$ invariant, and so is
tangent to $\cal G$ along $D_i$.  Thus the tangency locus of $\cal G$ and
$[w_j, w_k]$ consists of infinitely many
divisors and so $\cal G$ is closed under Lie bracket.

Since $\cal G$ is a $p$-closed foliation there exists a
purely inseparable morphism $\rho:X \rightarrow Y$ of degree $p^{n-2}$ such that
that $\ker d\rho = \TG$.

We therefore have an exact sequence
\[
\rho^*\Omega^1_Y \rightarrow \Omega^1_X \rightarrow \Omega^1_{\cal G}.
\]
Let $p:Y \dashrightarrow S$ be a general choice of a dominant rational map to a smooth surface
and let $U \subset Y$ and $V \subset S$ be two open sets such that $p_{|U}:U \rightarrow V$ is a morphism.
Since $p$ is general we may assume that $\rho^*p^*\Omega^1_V \rightarrow \rho^*\Omega^1_U$
is generically surjective.

Consider the dominant rational map $p \circ \rho:X \dashrightarrow S$.
Let $\overline{X}$ denote the closure of the graph of $p \circ \rho:\rho^{-1}(U) \rightarrow S$
in $X \times S$.  Then $\overline{X} \rightarrow X$ is birational (although $\overline{X}$ is
no longer smooth) and $p$ extends to a proper morphism $\overline{p}:\overline{X} \rightarrow S$.
Let $\overline{\cal F}$ and $\overline{\cal G}$
denote the transforms of $\cal F$ and $\cal G$ respectively.

Consider the Stein factorization of $\overline{p}$
\begin{center}
\begin{tikzcd}
\overline{X} \arrow{dr}{\overline{p}} \arrow{d}{\phi} & \\
T \arrow{r}{a} & S.
\end{tikzcd}
\end{center}
Let $\widetilde{T} \rightarrow T$ be a resolution of singularities (which exists since $T$ is a surface)
and let $\widetilde{X}$ be the normalization of the main component of $\overline{X}\times_T\widetilde{T}$.
Replacing $\overline{X}$ by $\widetilde{X}$ and $T$ by $\widetilde{T}$ we may freely assume that
$T$ is smooth.

Consider the composition of sheaf morphisms
\[
\Phi: \phi^*\Omega^1_T \rightarrow \Omega^1_{\overline{X}} \rightarrow \Omega^1_{\overline{{\cal F}}}.
\]
Since $\text{im}(\phi^*\Omega^1_T \rightarrow \Omega^1_{\overline{X}})$ is contained in
the kernel of $\Omega^{1}_{\overline{X}} \rightarrow \Omega^1_{\overline{\cal G}}$
we see that $\ker(\Phi)$ is a rank 1 reflexive subsheaf of $\phi^*\Omega^1_T$.
Pushing forward we have
$$0 \rightarrow \phi_*\ker(\Phi) \rightarrow \phi_*\phi^*\Omega^1_T = \Omega^1_T$$
where the latter equality holds since $\Omega^1_T$ is locally free and
$\phi_*\cal O_{\overline{X}} = \cal O_T$. The sheaf
$\phi_*\ker(\Phi)$ defines a foliation by curves $\cal L$ on $T$ such that $\phi^*\cal L = \overline{\cal F}$
and we are done.
\qed

\section{Cone of curves}\label{S:cone}

\subsection{Setup}
We will now consider consider codimension one foliations on
singular threefolds. As customary in birational geometry, we will assume that the threefold is
normal and that the canonical sheaf of the foliation is $\mathbb Q$-Cartier.

Given any birational morphism $\pi: \widetilde{X} \rightarrow X$, from a normal $3$-fold $\widetilde{X}$
we get an induced foliation $\widetilde{\cal F}$ on $\tilde{X}$.
Thus, we can write $$K_{\widetilde{\cal F}} =
\pi^*\KF+ \sum a(E_i, \cal F)E_i,$$
We say that $\cal F$ has canonical singularities  if
$a(E_i, \cal F) \geq 0$,  for every exceptional divisor $E_i$ of an arbitrary birational
morphism  $\pi: \widetilde{X} \rightarrow X$ from any normal projective variety $\widetilde{X}$ to $X$.

\subsection{Cone theorem for codimension one foliations}
The result below is proved in \cite[Theorem 1.1]{Spicer17} for foliated pairs. To avoid
an extra layer of definitions, we will stick to ``classical'' foliated varieties. The interested
reader will have no difficulties to extend Theorem \ref{THM:cone} for foliated pairs.

\begin{theorem}\label{T:cone}
Let $X$ be a projective $\bb Q$-factorial and klt projective variety of dimension three
and $\cal F$ a codimension one foliation with non-dicritical foliation singularities.
Suppose $\cal F$ has canonical singularities.
Then
$$\overline{NE}(X) = \overline{NE}(X)_{\KF \geq 0}+\sum \bb{R}_
+ [L_i]$$
where $L_i$ are curves. Furthermore, either
$L_i$ is contained in
$\text{sing}(X)$, or
$L_i$ may be taken to be a rational curve with
$\KF\cdot L_i \geq -6.$  In particular, the $\KF$-negative
extremal rays are locally discrete
in the $K_{\cal F}<0$ portion of the cone.
\end{theorem}

\subsection{Classification of extremal rays}
The extremal rays detected by Theorem \ref{T:cone} are of three different types according
to the dimension of
\[
\text{loc}(R) = \{x \in X : x \in C \text{ such that } [C] \in R\}\, ,
\]
the locus of points belonging to a curve $C$ with class spanning the extremal ray $R$.
In the terminology of \cite[Definition 23]{Spicer17}, a $\KF$-negative  extremal ray can be of
one of the following types.
\begin{enumerate}
\item Fiber type when $\dim \text{loc}(R)=3$. In this case, the foliation is the pull-back of a foliation on a surface
under a rational map with rational fibers.
\item Divisorial type when $\dim \text{loc}(R)=2$. The irreducible components of the support of the divisor spanned by the curves with class generating
the extremal ray $R$ are invariant or quasi-invariant by $\cal F$.

\item Flipping type when $\dim \text{loc}(R)=1$. In this case, we will say that the curves with class generating $R$ are flipping curves. Each
extremal ray of flipping type is represented by a finite number of curves.
\end{enumerate}

\begin{lemma}\label{L:intersects}
Set up as in Theorem \ref{T:cone}.
Let $C$ be an irreducible curve such that $[C] \in \overline{NE}(X)$ spans an extremal ray of flipping type.
Then either
$C\cap \text{sing}(X) \neq \emptyset$ or $C \subset \text{sing}(\cal F)$
\end{lemma}
\begin{proof}
Suppose $C \cap \text{sing}(X) = \emptyset$. Shrinking $X$ to a neighborhood of $C$ by
\cite{Cano} we may perform a series of blow ups in foliation invariant centres
$\pi:(\widetilde{X}, \widetilde{\cal F}) \rightarrow (X, \cal F)$ so that $\widetilde{\cal F}$ has simple singularities.
However, since we only blow up in invariant centres $\pi$ must be crepant, i.e., $K_{\widetilde{\cal F}} = \pi^*K_{\cal F}$.
In particular the strict transform of $C$ is still $K_{\widetilde{\cal F}}$-negative in which case we may apply
\cite[Corollary 11.2]{Spicer17} to conclude that $C \subset \text{sing}(\cal F)$.
\end{proof}

We will also need the following easy observations.

\begin{lemma}
\label{L:2fibretypecontractions}
Set up as in Theorem \ref{T:cone}.
Suppose there are two $K_{\cal F}$-negative extremal rays $R_1$ and $R_2$ of fibre type.
Then $\cal F$ is algebraically integrable and the closure of every leaf of $\cal F$ is a rational surface.
\end{lemma}
\begin{proof}
Let $f_i:X \rightarrow B_i$ be the contractions associated to $f_i$.  If either $B_1$ or $B_2$ is a curve, then
since $f_i$ only contracts curves tangent to the foliation, we see that $\cal F$ is algebraically integrable.
So assume that $B_i$ is a surface.

There exists a foliation $\cal G_i$ on $B_i$ such that $\cal F = f_i^*\cal G_i$. Let $p \in B_2$ be a general
point and let $C = f_2^{-1}(p)$. $C$ is a rational curve tangent to $\cal F$ and so $f_1(C)$ is
a rational curve tangent to $\cal G_1$.  Since $p$ is general, we see through a general point
of $B_1$ there is a $\cal G_1$-invariant rational curve which implies that $\cal G_1$, and hence $\cal F$,
is algebraically integrable. We see that the closure of a general leaf is a $\bb P^1$-fibration over
$\bb P^1$, and so the closure of a general leaf (hence any leaf) is rational.
\end{proof}

\begin{lemma}
\label{L:infintelymany-1curves}
Let $M$ be a smooth projective surface.
Suppose that $M$ contains infinitely many $(-1)$-curves.
Then $M$ is rational.
\end{lemma}
\begin{proof}
If $D$ is a psef divisor on $M$ then by considering the Zariski decomposition $D = P+Z$
where $P$ is nef and $Z \geq 0$ we see that $D \cdot C <0$ for only finitely many curves $C$.
Thus $K_M$ is not psef.

We may run a $K_M$-MMP which terminates
in either $\bb P^2$ or a $\bb P^1$-bundle over a curve $C$. We are done if we can show that we must have
$C \cong \bb P^1$.

Suppose for sake of contradiction that $g(C) \geq 1$.  Then the image of the $C_i$ under the MMP
must be disjoint rational curves, which implies that there are infinitely many disjoint $(-1)$-curves
on $M$, a contradiction of the Hodge index theorem.  Thus $M$ is rational and we are done.
\end{proof}

\subsection{Proof of Theorem \ref{THM:cone}}
By Lemma \ref{L:2fibretypecontractions} if there is more than one extremal ray of fibre type then
$\cal F$ is algebraically integrable and there is nothing else to prove.
So we may assume there are infinitely many extremal rays either of divisorial type or infinitely many extremal rays
of flipping type.

Let $\pi:\widetilde{X} \rightarrow X$ be a resolution of singlarities of $X$.
Suppose there are infinitely many extremal rays $R_i$ of divisorial type.  Let $D_i = \text{loc}(R_i)$.
Suppose for sake of contradiction that infinitely many $D_i$ are invariant.  Without loss of generality we may assume
that $D_i$ is disjoint from $\text{sing}(X)$ for all $i$.
We apply Theorem \ref{T:Jouanolou}
to produce a map $f:\widetilde{X} \rightarrow C$ such that the $\pi^*D_i$ are contained in
fibres of $f$.  This implies that all but finitely many
$D_i$ are numerically equivalent, and hence $\{R_i\}$ is in fact a finite set, a contradiction.

Thus, infinitely many $D_i$ are quasi-invariant.  Let $C_i$
be a curve spanning $R_i$, again, without loss of generality we may assume that $C_i$ is disjoint for $\text{sing}(X)$ for all $i$.
By Theorem \ref{T:quasiJouanolou}
either $\cal F$ is algebraically integrable or there is a map to a
surface $f:\widetilde{X} \rightarrow S$ contracting $\pi_*^{-1}(C_i)$ for all $i$.  Notice that in the latter case since the
$\pi_*^{-1}(C_i)$ are all contained in the fibres
of $f$ we see that all but finitely many $C_i$ must be in the same numerical
equivalence class, a contradiction, thus $\cal F$ must be algebraically integrable.
If $M$ is the closure of a general leaf,
we see that $M\cap D_i \subset M$ is a $(-1)$-curve and so, by Lemma \ref{L:infintelymany-1curves},
$M$ is rational.

It remains to treat the case where there are infinitely many extremal rays of flipping type.
According to Lemma \ref{L:intersects}, each flipping curve
must intersect $\text{sing}(X)$ or be contained in $\text{sing}(\cal F)$.
Since $X$ has terminal singularities $\text{sing}(X)$ is a finite collection of points.
As the singular set $\text{sing}(\cal F)$ consists of finitely many components,  we may assume that
we have infinitely many $C_i$ passing through a single point $p \in \text{sing}(X)$.

According to \cite[Corollary 6.4]{Spicer17}, there exists a $\cal F$-invariant analytic subvariety $L$ containing $p$ which contains the $C_i$.
To prove Theorem \ref{THM:cone} it suffices to verify that $L$ is algebraic and rational.

Let $\mu:M \rightarrow L$ be the minimal resolution of $L$
and notice that we can write $\mu^*K_{\cal F} = K_M+\Delta$ where $\Delta \geq 0$.
Let $C'_i$ be the strict transform of $C_i$ under $\mu$.
Observe that $\text{supp}(\Delta)\subset \text{exc}(\mu)\cup \mu^{-1}(\text{sing}(\cal F))$
and so
throwing away finitely many $C'_i$ we may assume that $\Delta \cdot C'_i \geq 0$ for all $i$.

Since each $C_i$ is a flipping
curve we know that $(C'_i)^2 <0$ and $(K_M+\Delta)\cdot C'_i <0$.
By adjunction we know that
$$(K_M+\Delta+C'_i)\cdot C'_i \geq -2$$
in particular
$$(C'_i)^2 = -1$$
for all $i$.

Assume that $\mu$ is not an isomorphism at $p$ (the case where $\mu$ is an isomorphism is easier and can be handled
in a similar manner).
There exists some irreducible component $E$ of $\mu^{-1}(p)$
such that $C'_i \cap E \neq \emptyset$ for infinitely many $i$.
Let $N = -2E^2+1$ and let
$\Sigma = 2E+\sum_{i = 1}^N C'_i$.
We claim that $N_{\Sigma/M}$ is ample. Indeed
for all $1 \leq j \leq N$
$$\Sigma \cdot C'_j  = 2E\cdot C'_j +\sum_{i = 1}^N C'_i\cdot C'_j
\geq 2E\cdot C'_j +(C'_j)^2 \geq 2+(-1) \geq 1$$
and
$$\Sigma \cdot E = 2E^2 +\sum_{i = 1}^NC'_i\cdot E \geq 2E^2 +N \geq 1$$
which proves our claim.

If $\widehat{M}$ is the formal completion of $M$ along $\Sigma$ then \cite[Theorem 6.7]{Hartshorne68}
implies that the field $\mathbb C(\widehat{M})$ of meromorphic functions on $\widehat{M}$ is of transcendence degree
$2$ over $\bb C$.  Since $\mathbb C(\widehat{M})$ is a field extension of $\mathbb C(\overline{L}^{Zar})$ where
$\overline{L}^{Zar}$
is the Zariski closure of $L$ we see that $\overline{L}^{Zar}$ is an algebraic surface.
It remains to show that $\overline{L}^{Zar}$ is rational.
Indeed, as above we see that the minimal resolution of
$\overline{L}^{Zar}$ contains infinitely many $(-1)$-curves in which case
we apply Lemma \ref{L:infintelymany-1curves} to conclude.
\qed

\bibliography{math}
\bibliographystyle{alpha}
\end{document}